\documentclass[12pt,a4paper,reqno]{amsart}
\usepackage{latexsym}
\usepackage{amssymb}
\usepackage{enumitem}
\usepackage[active]{srcltx}

\def \za{\alpha}

\def \ze{\varepsilon}

\def \zh{\theta}

\def \zl{\lambda}
\def \zm{\mu}

\def \zx{\xi}

\def \zp{\pi}

\def \zr{\rho}

\def \zt{\tau}

\def \zf{\varphi}

\def \zq{\psi}
\def \zw{\omega}

\def \zF{\Phi}

\def \zlma{\ell}

\def \zsu{\sum}

\def \zil{\int}

\def \zin{\cap}

\def \zun{\cup}
\def \zung{\bigcup}
\def \zex{\wedge}

\def \zmm{\pm}

\def \zpu{\cdot}
\def \zpor{\times}
\def \zci{\circ}

\def \zmei{\leq}
\def \zmai{\geq}
\def \zco{\subset}

\def \zpe{\in}

\def \zeq{\equiv}

\def \znoi{\neq}

\def \zpar{\partial}
\def \zinf{\infty}

\def \zfl{\rightarrow}

\def \zbv{\mid}

\def \z/{\over}

\newtheorem{theorem}{Theorem}[section]
\newtheorem{proposition}[theorem]{Proposition}
\newtheorem{corollary}[theorem]{Corollary}
\newtheorem{lemma}[theorem]{Lemma}

\newtheorem{remark}[theorem]{Remark}
\newtheorem{example}[theorem]{Example}

\newtheorem*{theorem*}{Theorem A}
\newtheorem*{corollary*}{Corollary}

\newcommand{\Aut}{\operatorname{Aut}}
\newcommand{\Diff}{\operatorname{Diff}}
\newcommand{\GL}{\operatorname{GL}}

\renewcommand{\L}{\mathcal{L}}

\title{Smooth toric actions are described by a single
vector field}

\author{F.J.~Turiel}
\address[F.J.~Turiel]{
Departamento de {\'A}lgebra, Geometr{\'\i}a y Topolog{\'\i}a,
Facultad de Ciencias,
Campus de Teatinos, s/n,
29071-M{\'a}laga, Spain}
\email[F.J.~Turiel]{turiel@uma.es}

\author{A.~Viruel}
\address[A.~Viruel]{
Departamento de {\'A}lgebra, Geometr{\'\i}a y Topolog{\'\i}a,
Facultad de Ciencias,
Campus de Teatinos, s/n,
29071-M{\'a}laga, Spain}
\email[A.~Viruel]{viruel@uma.es}

\thanks{Both authors are partially supported by
MEC-FEDER grant MTM2013-41768-P, and JA grants FQM-213. Second author is supported by XG grant EM2013/016.}

\begin{document}

\begin{abstract}
Consider a smooth effective action of  a torus $\mathbb{T}^n$ on a connected
$C^{\zinf}$-manifold $M$ of dimension $m$. Then $n\zmei m$.
In this work we show that if $n<m$, then there exist a complete vector
field  $X$ on $M$ such that the automorphism group of $X$ equals
$\mathbb T^n \zpor \mathbb{R}$, where the factor $\mathbb{R}$ comes
from the flow of $X$ and $\mathbb T^n$ is regarded as
a  subgroup of $\Diff(M)$.
\end{abstract}

\maketitle

\section{Introduction}\label{secI}

In a previous work \cite{TV}, and related to the
so called inverse Galois problem, we raised the question of whether or not a given
effective group action on a manifold is determined, or ``described", by non-classic
tensors in general, or more specifically, by vector fields. More
precisely: Consider an effective action of a Lie group $G$ on an $m$-manifold $M$,
thus we can think of $G$ as a subgroup of the group
$\Diff(M)$ of diffeomorphisms of $M$. Given a vector field $X$ on $M$, we say $X$ is a {\it describing vector field} for the $G$-action if the following
hold:
\begin{enumerate}[label={\rm (\arabic{*})}]
\item $X$ is complete and its flow $\zF_t$ commutes with the action
of $G$; so $G\leq \Aut(X)$.
\item\label{(2)} The group homomorphism
$$\begin{array}{ccc}
G\zpor{\mathbb R}&\rightarrow& \Aut(X)\\
(g,t)&\mapsto&g\zci\zF_{t}
\end{array}$$
is an isomorphism.
\end{enumerate}

Notice that we compare $\Aut(X)$ with $G\zpor{\mathbb R}$ instead of $G$ since we always have to take into account the flow of $X$ (see Remark \ref{remA}).

Within this setting, the main result in \cite{TV} shows that
any finite group action on a connected
manifold admits a describing vector field.
Here we extend this result to toric actions.

\begin{theorem*}
Consider an effective action of the torus $\mathbb T^n$ on a
 connected $m$-manifold $M$. Assume that $n\zmei m-1$.
Then there exist a describing vector field for this action.
\end{theorem*}

The proof of the Theorem A involves three steps. First, the result is
established for free actions (Section \ref{secB}) and then extended to
effective ones (Section \ref{secC}), in both cases assuming $m-n\zmai 2$.
Finally the case $m-n=1$ is considered in Section  \ref{secD}.

\begin{remark}\label{remA}
{\rm
Notice that according to Proposition \ref{proAA},
if the $\mathbb T^n$-action on $M^m$ is effective then $n\zmei m$.

On one hand, when $n=m$, we can make the identification $M=\mathbb T^n$ endowed with
the natural  $\mathbb T^n$-action. In this case if $X$ is the fundamental vector field associated
to a dense affine vector field on $\mathbb T^n$ (see Section \ref{secA}), then
$\Aut(X)=\mathbb T^n$.

On the other hand, when $n<m$ no complete vector field $X$ on $M$ verifies
$\Aut(X)=\mathbb T^n$. Indeed, assume $\Aut(X)=\mathbb T^n$, and let
$X_1 ,\dots,X_n$ be a basis of the Lie algebra of fundamental vector fields
and $f_1 ,\dots,f_n$ any $\mathbb T^n$-invariant functions. Then
$[X_r ,\zsu_{j=1}^n f_j X_j ]=0$, $r=1,\dots, n$; so the flow of
$\zsu_{j=1}^n f_j X_j$ commutes with the action of $\mathbb T^n$ and,
as every element of the flow of $X$ belongs to $\Aut(X)$, with this flow too.
That is to say the  flow of $\zsu_{j=1}^n f_j X_j$ is included in $\Aut(X)$
and, necessarily, $\Aut(X)\znoi \mathbb T^n$ {\it contradiction}.

Thus our result is ``minimal" because the flow of $X$ is always included in $\Aut(X)$.}
\end{remark}

\begin{remark}\label{remB}
{\rm
Finally, notice that Theorem A cannot be extended to a general compact Lie group. In Section \ref{secE} we construct effective actions of
$SO(3)$ (Example \ref{ejeB}), and  of a non-connected
compact group of dimension two (Example \ref{ejeC}), for which there is no describing vector field.
}
\end{remark}

\noindent \textbf{Terminology:} The reader is supposed to be familiarized with our previous
paper  \cite{TV}. All structures and objects considered are real
$C^{\zinf}$ and manifolds are without boundary, unless another thing
is stated. For the general questions on Differential Geometry the reader is
referred to \cite{KN} and for those on Differential Topology to
\cite{HI}.
\medskip

\noindent \textbf{Acknowledgements:} The authors would like to thank
Prof.\ Arthur Wasserman for suggesting that our original result on finite
group actions could be extended to $S^1$-actions,
and for his helpful comments on the development of this work.

\section{Preliminary results on vector fields}\label{secA}
In this section we collect some results on vector fields that are needed in the following sections.
On $\mathbb R^k$ we set coordinates $x=(x_1 ,\dots,x_k )$, and define
 $\zx=\zsu_{j=1}^{k}x_{j}\zpar/\zpar x_{j}$.

\begin{lemma}\label{lemA}
For any function $g\colon \mathbb R^k \zfl\mathbb R$ with $g(0)=0$
there is a function $f\colon \mathbb R^k \zfl\mathbb R$ such that
 $\zx\zpu f=g$.
\end{lemma}
\begin{proof}
Although the proof of this result is routine we outline it here. Let
$\mathcal T$ denote the Taylor's series of $g$ at $0\zpe\mathbb R^k$.
Then there exists a series $\mathcal S$ such that formally
$\zx\zpu\mathcal S=\mathcal T$. According to Borel's Theorem \cite[Theorem 1.5.4]{NA}
there is a function $\zf$ whose Taylor's series at origin
equals $\mathcal S$, and therefore making $g-\zx\zpu\zf$ we may
suppose $\mathcal T=0$.

Since $\zx$ is hyperbolic at origin, following \cite[Theorem 10, page 38]{RRO}
there exist a function
$\tilde f\colon\mathbb R^k \zfl\mathbb R$ such that $g-\zx\zpu\tilde f$
vanishes on a open neighborhood $A$ of $0\zpe\mathbb R^k$.
Therefore it suffices to show the result when $g_{\zbv A}=0$.

But $\mathbb R^k -\{0\}$ can be identified to $S^{k-1}\zpor\mathbb R$
in such a way that $\zx=\zpar/\zpar t$ where $t$ is the variable in
$\mathbb R$ and $S^{k-1}\zpor(-\zinf,1)$ corresponds to
$B_\ze (0)-\{0\}\zco A$ for some radius $\ze>0$.

Finally set $f=\zil_{0}^{t}g{\operatorname d}s$ on $S^{k-1}\zpor\mathbb R$
and $f(0)=0$.
\end{proof}

Recall that a vector field $T$ on $\mathbb T^n$, endowed with coordinates
$\zh=(\zh_1 ,\dots,\zh_n )$, is named {\it affine} if
$T=\zsu_{r=1}^{n}a_{r}\zpar/\zpar \zh_{r}$ where
$a_1 ,\dots,a_n \zpe \mathbb R$. Then, the trajectories of an affine
vector field are dense if and only if $a_1 ,\dots,a_n$ are rationally
independent; {\it in this case we say that $T$ is dense}.

\begin{lemma}\label{lemB}
On $\mathbb R^k \zpor \mathbb T^n$ with coordinates
$(x,\zh)=(x_1 ,\dots,x_k ,\zh_1 ,\dots,\zh_n )$ consider the vector field
$X=\zx+T$, where $T$ is dense. Then $\L_X$, the set of all vector fields on $\mathbb R^k \zpor \mathbb T^n$
which commute with $X$, is a Lie algebra of dimension $k^2 +n$ with basis
$$\left\{ x_j {\frac {\zpar} {\zpar x_{\zlma}}},{\frac {\zpar}
{\zpar \zh_r}}\right\},\,\, j,\zlma=1,\dots,k;\, r=1,\dots,n.$$
\end{lemma}

\begin{proof}
Let $Y\in\L_X$ be the vector field defined as $Y=\zsu_{j=1}^k f_j (x,\zh)\zpar/\zpar x_j
+\zsu_{r=1}^n g_r (x,\zh)\zpar/\zpar \zh_r$, and let $\zF_t$
denote its flow. Since $\{0\}\zpor\mathbb T^n$ is compact,
$\zF_t$ is defined on $(\{0\}\zpor\mathbb T^n )\zpor(-\ze,\ze)$ for some
$\ze>0$.

Observe that $\{0\}\zpor\mathbb T^n$ is the set of all points whose
$X$-trajectory has compact adherence. Therefore
$\zF_t (\{0\}\zpor\mathbb T^n )\zco\{0\}\zpor\mathbb T^n$ for any
$t\zpe(-\ze,\ze)$, what implies that $Y$ is tangent to
$\{0\}\zpor\mathbb T^n$, and therefore $f_j (0,\zh)=0$, for $j=1,\dots,k$.

A computation, taking into account that every $\zpar/\zpar\zh_r$ commutes
with $X$, yields $$[X,Y]=\widetilde Y
+\zsu_{r=1}^n (X\zpu g_r )\zpar/\zpar \zh_r$$ where $\widetilde Y$ is a
functional combination of $\zpar/\zpar x_i$'s.
Therefore $X\zpu g_r =0$, i.e.\ $g_r$ is constant
along the trajectories of $X$, for $r=1,\dots,n$. But the $\zw$-limit of these trajectories is
$\{0\}\zpor\mathbb T^n$ and $g_r$ is constant on this set because
$\{0\}\zpor\mathbb T^n$ is the adherence of the trajectory of any of its
points, so each  $g_r$ is constant on
$\mathbb R^k \zpor \mathbb T^n$.

By considering $Y-\zsu_{r=1}^n g_r \zpar/\zpar \zh_r$ one may suppose
$Y=\zsu_{j=1}^k f_j (x,\zh)\zpar/\zpar x_j$ where each
$f_j (\{0\}\zpor\mathbb T^n)=0$. Now from $[X,Y]=0$ follows
$X\zpu f_j =f_j$, $j=1,\dots,k$.

On the other hand given $f\colon\mathbb R^k \zpor\mathbb T^n
\zfl\mathbb R$ such that $f(\{0\}\zpor\mathbb T^n)=0$ and $X\zpu f=f$,
then $f$ does not depend on $\zh$ and it is linear on $x$. Indeed if $k=0$
it is obvious; assume $k=1$ by the moment. Then $f=xg$ for some function
$g$ since $f$ vanishes on $\{0\}\zpor\mathbb T^n$, and $X\zpu f=f$
becomes $X\zpu g=0$. Therefore $g$ is constant along the trajectories of $X$
and, by the same reason as before, constant on
$\mathbb R\zpor\mathbb T^n$.

Now suppose $k\zmai 2$. Let $E$ be any vector line in $\mathbb R^k$. As
$X$ is tangent to $E\zpor\mathbb T^n$ and the restriction of $\zx$ to $E$
is still the radial vector field, $f\colon E\zpor\mathbb T^n \zfl\mathbb R$ is
independent of $\zh$ and linear on $E$ (it is just the case $k=1$).
Since the union of all the vector lines
$E$ equals $\mathbb R^k$, it follows that $f$ does not depend on $\zh$ and
$f\colon\mathbb R^k \zfl \mathbb R$ is linear on each $E$. But, as it is well
known, this last property implies that $f\colon\mathbb R^k \zfl \mathbb R$
is linear.

In short every $f_j$ is linear on $x$ and independent of $\zh$.
\end{proof}

\begin{corollary}\label{cor_de_2.2}
If $F\colon \mathbb R^k \zpor \mathbb T^n \zfl
 \mathbb R^k \zpor \mathbb T^n$ is an automorphism of
$X=\zx+T$ and $T$ is dense, then there exist an isomorphism $\zf\colon\mathbb R^k
\zfl\mathbb R^k$ , and an element $\zl\zpe \mathbb T^n$ such that
$F(x,\zh)=(\zf(x),\zh+\zl)$.
\end{corollary}

\begin{proof}
The diffeomorphism $F$ induces an isomorphism on $\L_X$, the Lie algebra described in Lemma~\ref{lemB}. Now observe that:
\begin{enumerate}
\item The
only elements in $\L_X$ with singularities are the elements $\zsu_{j,\zlma=1}^k a_{j\zlma}x_j\zpar/\zpar x_{\zlma}$, where
$(a_{j\zlma})\zpe\GL(n,\mathbb
R)$. Besides they give rise to the foliation
$d\zh_1 =\dots=d\zh_n =0$ (extend it to $\{0\}\zpor\mathbb T^n$ by
continuity); so this foliation is an invariant of $F$.

\item The center of $\L_X$
is spanned by $\zx, \zpar/\zpar\zh_1 ,
\dots,\zpar/\zpar\zh_r$. But the adherences of the trajectories of a vector
field  $b\zx+b_1 \zpar/\zpar\zh_1 +\dots+b_r \zpar/\zpar\zh_r$
are always tori if
and only if $b=0$, so $F$ sends the Lie subalgebra
spanned by $ \zpar/\zpar\zh_1 ,\dots,\zpar/\zpar\zh_r$ into itself.
\end{enumerate}
These two facts imply that $F(x,\zh)=(\zf(x),\zq(\zh))$ where
$\zq\colon\mathbb T^n \zfl\mathbb T^n$ is an affine transformation
of $\mathbb T^n$; that is
$$\zq(\zh)=\left(\zsu_{j=1}^{n}c_{1j}\zh_{j},\dots,
\zsu_{j=1}^{n}c_{nj}\zh_{j})\right)+\zl$$
where $(c_{\zlma j})\zpe\GL(n,\mathbb Z)$ and $\zl\zpe\mathbb T^n$.
As $X=T$ on $\{0\}\zpor\mathbb T^n$, which is an invariant of $F$,
it follows  that $(a_1 ,\dots,a_n )$ is an eigenvector
of $(c_{\zlma j})$ whose eigenvalue
equals 1. But in this case $a_1 ,\dots,a_n$ are rationally dependent unless
$(c_{\zlma j})=Id$. In short $\zq(\zh)=\zh+\zl$.

On the other hand $\zf\colon\mathbb R^k \zfl\mathbb R^k$ has to be an automorphism of $\zx$, which implies that $\zf$ is an isomorphism
\cite[Lemma 3.4]{TV}.
\end{proof}

\begin{lemma}\label{lemC}
On $\mathbb R^k \zpor \mathbb T^n$ one considers the vector field
$\widetilde X =\tilde\zx +\widetilde T$ where
$\tilde\zx=\zsu_{j=1}^{k}\tilde f_{j}(x)\zpar/\zpar x_{j}$ and
$\widetilde T=\zsu_{i=1}^{n}\tilde g_{r}(x)\zpar/\zpar \zh_{r}$.
Assume that on $\mathbb R^k$ the following hold:
\begin{enumerate}[label={\rm (\alph{*})}]
\item $\tilde\zx$ is complete,
\item $\tilde\zx(0)=0$ and its linear part at the origin is a positive
multiple of identity,
\item the outset of the origin equals $\mathbb R^k$.
\end{enumerate}

Then there exists a self-diffeomorphism of $\mathbb R^k \zpor \mathbb T^n$
that commutes with the natural $\mathbb T^n$-action and
transforms $\widetilde X$ into
$$b\zx+\zsu_{r=1}^{n}b_{r} {\frac {\zpar} {\zpar \zh_r}}$$
where $b\zpe\mathbb R^+$ and $b_1 ,\dots,b_n \zpe\mathbb R$.
\end{lemma}

\begin{proof}
By Sternberg's linearization theorem \cite{SST},
see \cite[Proposition 2.1]{TV}, there exists a diffeomorphism $f\colon\mathbb R^k \zfl\mathbb R^k$
transforming $\tilde\zx$ into $b\zsu_{j=1}^k x_j \zpar/\zpar x_j$ with $b>0$.  Dividing by $b$ we may suppose
$\tilde\zx=\zx$.

By Lemma \ref{lemA} there are functions $\zf_1 ,\dots,\zf_n \colon\mathbb
R^k \zfl\mathbb R$ such that $\zx\zpu\zf_r =g_r -g_r (0)$, $r=0,\dots,n$.
Now, if $\zf=(\zf_1 ,\dots,\zf_n )$ and
$\tilde\zp\colon\mathbb R^n \zfl\mathbb T^n$ is the canonical covering, then the diffeomorphism $F\colon \mathbb R^k \zpor \mathbb T^n \zfl
 \mathbb R^k \zpor \mathbb T^n$ given by $F(x,\zh)
=(x,\zh-\tilde\zp\zci\zf)$,
transforms $\widetilde X$ into $\zx+\zsu_{r=1}^n g_r (0)\zpar/\zpar\zh_r$.
\end{proof}

\begin{remark}\label{remB2}
{\rm If $\widetilde X$ matches the hypotheses of Lemma \ref{lemC} and
$h\colon \mathbb R^k \zfl \mathbb R$ is a positive bounded function, then
$h\widetilde X$ satisfies these hypotheses too (when $h$ is regarded  as
 a function on  $\mathbb R^k \zpor \mathbb T^n$ in the obvious way).
Therefore $h\widetilde X$ can be written as in Lemma \ref{lemC} for a
suitable choice of coordinates, and thus the control of its automorphisms
becomes simple.}
\end{remark}

\section{Free actions}\label{secB}
In this section we assume there is a free $\mathbb T^n$-action on the
connected $m$-manifold $M$, where $m-n\zmai 2$. This gives rise to a principal fibre
bundle $\zp\colon M\zfl B$ whose structure group is $\mathbb T^n$ and
$B$ is a connected manifold of dimension $k=m-n\zmai 2$. Then, we construct
a suitable vector field on $B$ that is later on lifted
to $M$ by means of a connection.

In order to construct the vector field on $B$, we closely follow along the lines in
 \cite[Section 3]{TV} applied to the case of the trivial group action on $B$.

Consider a Morse function $\zm\colon B\zfl{\mathbb R}$ that is
proper and non-negative. Denote by $C$ the set of its
critical points, which is closed and discrete, and therefore countable. As $B$ is
paracompact, there exists a locally finite family $\{A_{p}\}_{p\zpe C}$
of disjoint open set such that $p\zpe A_{p}$, $p\zpe C$.

Following along the lines in \cite[Section 3]{TV}, there exist a
Riemannian metric ${\tilde g}$ on $B$ such that the gradient vector
field $Y$ of $\zm$ is complete and, besides, around each $p\zpe C$
there are coordinates $(x_{1},\ldots,x_{k})$ with $p\zeq 0$ and
$Y=\zsu_{j=1}^{k}\zl_{j}x_{j}\zpar/\zpar x_{j}$,
$\zl_1 ,\dots,\zl_k \zpe\mathbb R -\{0\}$, where

\begin{enumerate}[label={\rm (\arabic{*})}]
\item $\zl_1 =\dots=\zl_k >0$  if $p$ is a source of $Y$, that is a
 minimum of $\zm$,
\item $\zl_1 =\dots=\zl_k <0$  if $p$ is a sink of $Y$ (a maximum
of $\zm$),
\item some $\zl_j$ are positive and the remainder negative if $p$ is
a saddle.
\end{enumerate}

Note that these properties still hold when $Y$ is multiplied by a positive
bounded function, since they only depend on the
Sternberg's Theorem.

Let $I$ be the set of local minima of $\zm$, that is the set of
sources of $Y$, and $S_{i}$, $i\zpe I$, the outset of $i$ relative
to $Y$. Now Lemma 3.3 of \cite{TV} becomes:

\begin{lemma}\label{lemD}
The family $\{ S_{i}\}_{i\zpe I}$ is locally finite and the
set $\zung_{i\zpe I}S_{i}$ is dense in $B$.
\end{lemma}

In what follows, and by technical reasons, one makes use of the notion of
\emph{order of nullity} instead of \emph{chain}. More exactly, for every $i\zpe I$ we
choose a subset $P_i$ of $A_i$ with $k+1$ points close enough to $i$
but different from it, in such a way that the linear $\za$-limits of their
trajectories are in general position (see \cite[pags.\ 319 and 320]{TV}
for definitions).

Set $P=\zung_{i\zpe I}P_i$. Consider an injective map
$p\zpe P\mapsto n_p \zpe\mathbb N -\{0\}$. Let $\mathbb N'$ be the
image of $P$.

By definition a differentiable object has {\it order $r$ at point $q$} if its
$(r-1)$-jet at this point vanishes but its $r$-jet does not; for instance
$Y$ has order one at sources, sinks and saddles.

Since $\{A_i \}_{i\zpe I}$ is still locally finite, one may construct a bounded
function $\zt\colon B\zfl\mathbb R$ such that $\zt$ is positive on $B-P$
and has order $2n_p$ at every $p\zpe P$.

Put $Z=\zt Y$. Then $Z^{-1}(0)=Y^{-1}(0)\zun P$; that is the singularities
of $Z$ are the sources, sinks and saddles of $Y$ plus the points of $P$,
that we call {\it artificial singularities} and whose order is even $\zmai 2$.
Note that two different artificial singularities have different orders.

Let $R_i$, $i\zpe I$, be the $Z$-outset of $i$. As $S_i -R_i$ is the union of
$k+1$ half-trajectories of $Y$ one has:

\begin{lemma}\label{lemE}
The family $\{ R_{i}\}_{i\zpe I}$ is locally finite and the
set $\zung_{i\zpe I}R_{i}$ is dense in $B$.
\end{lemma}

On the principal fibre bundle $\zp\colon M\zfl B$ consider a connection
$\mathcal C$ which is a product around every fibre $\zp^{-1}(p)$,
$p\zpe C$ (that is there exist an open set $p\zpe A\zco B$ and a fiber bundle
isomorphism between  $\zp\colon \zp^{-1}(A)\zfl A$
and $\zp_1 \colon A\zpor\mathbb T^n \zfl A$ in such a way that
$\mathcal C$, regarded on $\zp_1 \colon A\zpor\mathbb T^n \zfl A$,
is given by $\mathcal C (q,\zh)=T_q A\zpor \{0\}\zco
T_{(q,\zh)}(A\zpor\mathbb T^n )$).
This kind of connection always exists because $\{ A_{p}\}_{p\zpe C}$ is
 locally finite.

Let $Y'$ denote the lift of $Y$ to $M$ by means of $\mathcal C$; that
is $Y'(u)\zpe\mathcal C (u)$ and $\zp_{*}(Y'(u))=Y(\zp(u))$ for every
$u\zpe M$. By construction $Y'$ is $\mathbb T^n$-invariant and
 $Y'(u)=0$ if and only if $Y(\zp(u))=0$.

Let $T$ be a dense affine vector field on $\mathbb T^n$ and $T'$ the
fundamental vector field, on $M$, associated to $T$ through the action. As
describing vector field we take $X'=(\zt\zci\zp)(Y'+T')$, which clearly is
$\mathbb T^n$-invariant and complete.

The remainder of this section is devoted to show that $X'$ is a
describing vector field. First we study the behavior of $X'$ near some
fibres. If $p$ is a source of $Y$ there exist coordinates $(x_1 ,\dots,
x_k )$, about $p\zpe B$, with $p\zeq 0$ and
$Y=a\zsu_{j=1}^{k}x_{j}\zpar/\zpar x_{j}$, $a>0$. As around
$\zp^{-1}(p)$ the connection is a product, these coordinates can be
prolonged to a system of coordinates $(x,\zh)$ on a product open set
$A\zpor\mathbb T^n$, with the obvious identifications, while
$\mathcal C$ is given by the first factor. In this case
$$Y'+T'=a\zsu_{j=1}^{k}x_{j}{\frac {\zpar} {\zpar x_j}}+T$$
since $T'$ is just $T$ regarded as a vector field on
$A\zpor\mathbb T^n$.

The same happens when $p$ is a sink but $a<0$. If $p$ is a saddle then
the model of the first part is
$\zsu_{j=1}^{k}\zl_{j}x_{j}\zpar/\zpar x_{j}$ with some $\zl_j$ positive
and the others negative.

Thus, when $p\zpe C$, the torus $\zp^{-1}(p)$  is the adherence of a
trajectory of $X'$, this vector field never vanishes on $\zp^{-1}(p)$ and,
besides:
\begin{enumerate}[label={\rm (\alph{*})}]
\item If $p\zpe I$, then $\zp^{-1}(p)$ is the $\za$-limit of some external
trajectories but never the $\zw$-limit.
\item If $p$ is a sink, then $\zp^{-1}(p)$ is the $\zw$-limit of some external
trajectories but never the $\za$-limit.
\item If $p$ is a saddle, then $\zp^{-1}(p)$ is the $\za$-limit of some external
trajectories and the $\zw$-limit of other ones.
\end{enumerate}

On the other hand $(X')^{-1}(0)=\zp^{-1}(P)$. Moreover if $p\zpe P$
then $X'$ has order $2n_p$ at each point of $\zp^ {-1}(p)$.

If $X'$ is multiplied by a positive, bounded and $\mathbb T^n$-invariant
function $\zr\colon M\zfl\mathbb R$, then $\zr=\tilde\zr\zci\zp$ for some
positive and bounded function $\tilde\zr\colon B\zfl\mathbb R$. Thus
$\zr X'= ((\tilde\zr \zt)\zci\zp)(Y'+T')$. As $\zt$ and $\tilde\zr \zt$ have
 the same  essential properties, the foregoing description still holds for
$\zr X'$. In other words, this description is geometric and independent
of how trajectories are parameterized.

Consider $i\zpe I$ and identify its $Y$-outset $S_i$ to $\mathbb R^k$ in
such a way that $i\zeq 0$ and
$Y=a\zsu_{\zlma=1}^{k}x_{\zlma}\zpar/\zpar x_{\zlma}$, $a>0$. As $S_i$
is contractible the fibre bundle $\zp\colon\zp^{-1}(S_i )\zfl S_i$ is trivial;
so it can be regarded like
$\zp_1 \colon\mathbb R^k \zpor\mathbb T^n\zfl\mathbb R^k$ while
$$Y'+T'=a\zsu_{\zlma=1}^{k}x_{\zlma}{\frac {\zpar} {\zpar x_{\zlma}}}
+\zsu_{r=1}^{n}g_{r}(x){\frac {\zpar} {\zpar \zh_r}}.$$

Finally Lemma \ref{lemC} allows us to suppose
$$Y'+T'=a\zsu_{\zlma=1}^{k}x_{\zlma}{\frac {\zpar} {\zpar x_{\zlma}}}
+\zsu_{r=1}^{n}a_{r}{\frac {\zpar} {\zpar \zh_r}}$$
with $a>0$ and $a_1 ,\dots,a_n \zpe\mathbb R$. Moreover since
$T'=\zsu_{r=1}^{n}a_{r}\zpar/\zpar \zh_{r}$ at any point of
$\{0\}\zpor\mathbb T^n$, scalars $a_1 ,\dots,a_n$ are rationally
independent (recall that $Y'(u)=0$ whenever $Y(\zp(u))=0$).

Now is clear that, for every $p\zpe P_i$ and $u\zpe \zp^{-1}(p)$, there
exists a trajectory of $X'=(\zt\zci\zp)(Y'+T')$ whose $\za$-limit and $\zw$-limit are
$\zp^{-1}(i)$ and  $u$ respectively. As $n_p \znoi n_{p'}$ when
$p\znoi p'$, the existence of this kind of trajectories shows that any
automorphism of $X'$ has to send $\zp^{-1}(i)$ in itself and, consequently,
the $X'$-outset of $\zp^{-1}(i)$ in itself too (here outset means the set
of points whose trajectory has its $\za$-limit included in $\zp^{-1}(i)$).

The next question is to determine this outset. First we identify the outset
$R_i$, of $i$ with respect to $Z=\zt Y$, to $\mathbb R^k$ in such a way
that $i\zeq 0$ and $Z=b\zsu_{\zlma=1}^{k}x_{\zlma}\zpar/\zpar
x_{\zlma}$, $b>0$. Again  $\zp\colon\zp^{-1}(R_i )\zfl R_i$ is trivial and
reasoning as before, taking into a account that $X'$ is complete on
$\zp^{-1}(R_i )$ which allows to apply Lemma \ref{lemC}, leads us to the
case $\zp^{-1}(R_i )\zeq\mathbb R^k \zpor \mathbb T^n
\zfl\mathbb R^k \zeq R_i$ and
$$X'=b\zsu_{\zlma=1}^{k}x_{\zlma}{\frac {\zpar} {\zpar x_{\zlma}}}
+\zsu_{r=1}^{n}b_{r}{\frac {\zpar} {\zpar \zh_r}}$$
where $b>0$ and $b_1 ,\dots,b_n \zpe \mathbb R$ are rationally
independent. Therefore the $X'$-outset of $\zp^{-1}(i)$ equals
$\zp^{-1}(R_ i)$.

Let $F\colon M\zfl M$ be an automorphism of $X'$. Then
$F\colon\zp^{-1}(R_ i)\zfl \zp^{-1}(R_ i)$ is a diffeomorphism. As the
trajectories of $\zsu_{r=1}^{n}b_{r}\zpar/\zpar \zh_{r}$ are dense, from
Corollary \ref{cor_de_2.2} it follows that
$F(x,\zh)=(\zf(x),\zh+\zl)$, where $\zl\zpe \mathbb T^n$ and
$\zf\colon\mathbb R^k\zfl\mathbb R^k$ is an isomorphism.

For any $p\zpe P_i$ some trajectories of $X'$ with $\za$-limit
$\zp^{-1}( i)$ have, as $\zw$-limit, a point of $\zp^{-1}(p)$, that is a
singularity of order $2n_p$. Since $n_p \znoi n_{p'}$ when $p\znoi p'$,
the set of these trajectories has to be an invariant of $F$.
Regarded on $R_i \zco B$, and taking into account that $X'$ projects in
$Z$, this fact implies that $\zf$ has to map the trajectory of
$Z=b\zsu_{\zlma=1}^{k}x_{\zlma}\zpar/\zpar x_{\zlma}$
of $\za$-limit $i$ and $\zw$-limit $p$ into itself. Thus the direction vector of this
curve is an eigenvector of $\zf$ with positive eigenvalue. But there are
$k+1$ eigenvector (as many as points in $P_i$) and they are in general
position, so $\zf$ is a positive multiple of identity.

Therefore there exists $t_i$ such that
$\zF_{t_i}(x,\zh)=(\zf(x),\zh+\tilde\zl_i)$ where $\zF_{t}$ denotes the
flow of $X'$ and $\tilde\zl_i \zpe\mathbb T^n$. Since $\zF_{t}$ and the
action of $\mathbb T^n$ commute, that shows the existence of
$t_i \zpe\mathbb R$ and $\zl_i \zpe\mathbb T^n$ such that
$F=\zl_i \zci\zF_{t_i}$ on $\zp^{-1}(R_i )$.

It easily seen that the family $\{\zp^{-1}( R_{i})\}_{i\zpe I}$ is locally
finite and the set $\zung_{i\zpe I}\zp^{-1}( R_{i})$ is dense in $M$.

On each $\zp^{-1}( R_{i})$ the action of $\mathbb T^n$ and $F$ commute,
so they do on $M$. Thus $F$ induces a diffeomorphism $f\colon B\zfl B$
such that $f\zci\zp=\zp\zci F$. Besides, since $Z$ is the projection of $X'$,
our $f$ is an automorphism of $Z$. Now from the expression of
$F\colon\zp^{-1}( R_{i})\zfl\zp^{-1}( R_{i})$ it follows that $f=\zf_{t_i}$ on
$R_i$, $i\zpe I$, where $\zf_t$ is the flow of $Z$.

\begin{lemma}\label{lemF}
All $t_i$'s are equal and $f=\zf_t$ for some $t\zpe\mathbb R$.
\end{lemma}

\begin{proof} Notice that $X$ has no regular periodic trajectories
(here dimension of $B$ $\zmai 2$ is needed). Then, the lemma follows from
the proof of Lemma 3.7 in \cite{TV}, when $G$ is the trivial group and $X$ becomes $Z$.
\end{proof}

Therefore composing $F$ with $\zF_{-t}$ allows to suppose that $f$ is
the identity and $F=\zl_i$ on each $\zp^{-1}( R_{i})$, $i\zpe I$.

Consider a $\mathbb T^n$-invariant Riemannian metric on $M$. Then
$F$ is an isometry on every $\zp^{-1}( R_{i})$ so, by continuity, on $M$.
Take $i_0 \zpe I$; then the isometries $F$ and $\zl_{i_0}$ agree on
$\zp^{-1}( R_{i_0})$. But on connected manifolds, isometries are
determined by their $1$-jet at any point. Therefore $F=\zl_{i_0}$
on $M$. In other words
$(\zl,t)\zpe \mathbb T^n\zpor{\mathbb R}\zfl \zl\zci\zF_t \zpe\Aut(X)$
is an epimorphism.

We now prove injectivity. Assume $\zl\zci\zF_t =Id$, that is
$\zF_t =(-\zl)$. Then $\zf_t\colon B\zfl B$ equals the identity because
$(-\zl)$ induces this map on $B$.
Since $Z$ has no periodic regular trajectories this implies $t=0$ and,
finally, $\zl=0$ because the action of $\mathbb T^n$ is effective.
{\it In short $X'$ is a describing vector field for free actions.}

\begin{remark}\label{remC}
{\rm
Note that if $\zr\colon M\zfl\mathbb R$ is a $\mathbb T^n$-invariant,
positive and bounded function, then $\zr X'$ is a describing vector field too.
Indeed, there exists a function $\tilde\zr\colon B\zfl\mathbb R$ such that
$\zr=\tilde\zr\zci\zp$ and it suffices reasoning with $\tilde\zr\zt$
instead of $\zt$.}
\end{remark}

\section{Effective actions}\label{secC}
Throughout this section we assume that $m-n\zmai 2$ and the
 $\mathbb T^n$-action is effective. Our next goal is to construct a describing vector
field on $M$. Let $S$ be the set of those
 points of $M$ whose isotropy (stabilizer)
is non trivial. By Proposition \ref{proAA}
 the set $M-S$ is dense, open, connected and $\mathbb T^n$-invariant. Moreover the
$\mathbb T^n$-action on $M-S$ is free, so we can consider on this set a
describing vector field $X'$ as in Section \ref{secB}. Let $\zf$ be a function
like in Proposition~\ref{proAB} for $X'$ and $M-S$, and
$\widehat X'$ be the vector field on $M$ given by $\zf X'$ on $M-S$ and
zero on $S$.

Set $\zq=h\zci\zf$ where $h\colon\mathbb R\zfl\mathbb R$ is defined
by $h(t)=0$ if $t\zmei 0$ and $h(t)=exp(-1/t)$ if $t>0$. The function
$\zq$, which is $\mathbb T^n$-invariant, vanishes at order infinity at
every point of $S$, i.e.\ all its $r$-jets vanish. Besides, it is bounded on
$M$ and positive on $M-S$.

Now put $\widetilde X=\zq\widehat X'$. Clearly $\widetilde X$
vanishes at order
infinity at  $u$ if and only if $u\zpe S$; therefore $S$ is an invariant of
$\widetilde X$. Moreover $\widetilde X$ is complete on $M$ and $M-S$
respectively. By Remark~\ref{remC} our $\widetilde X$ is a describing
vector field on $M-S$ since $\widetilde X=(\zq\zf)X'$ on this set.

If $F\colon M\zfl M$ is an automorphism of $\widetilde X$ then $F(S)=S$,
and  $F\colon M-S\zfl M-S$ is an automorphism of $\widetilde X$; so on
$M-S$ one has $F=\zl\zci\widetilde\zF_t$, where $\widetilde\zF_t$ is
the flow of $\widetilde X$. By continuity $F=\zl\zci\widetilde\zF_t$
everywhere, which implies that $\widetilde X$ is a describing vector field
on $M$ (the homomorphism injectivity is inherited from $M-S$).

\section{The case $m-n=1$}\label{secD}
First assume the action is free, which gives rise to a principal fibre bundle
$\zp\colon M\zfl B$ with $B$ connected and of dimension one. Therefore $B$
is $\mathbb R$ or $S^1$ and $\zp\colon M\zfl B$ can be identified to
$\zp_1 \colon B\zpor\mathbb T^n \zfl B$. One will need the following result
whose proof is routine.

\begin{lemma}\label{lemG}
On a open set $0\zpe A\zco\mathbb R$ consider a vector field $X$ such that
its $(r-1)$-jet at origin vanishes but its $r$-jet does not, $r\zmai 1$.
Let $\zf_t$ be the flow of $X$. Given $f\colon A\zfl\mathbb R$ and
$t_1 ,t_2 \zpe\mathbb R$, if $f=\zf_{t_1}$ on $A\zin (0,\zinf)$ and
$f=\zf_{t_2}$ on $A\zin (-\zinf,0)$ then $t_1 =t_2$.
\end{lemma}

Set $Y=q(q^2 +1)^{-1}\zpar/\zpar x$, where
$q=x(x-1)(x-2)(x-3)(x-4)$ when $B=\mathbb R$, and
$Y=\sin(3\za)\zpar/\zpar\za$ if $B=S^1$ endowed with the angular
coordinate $\za$. Clearly $Y$ is complete. In the first case the sources are
$0,2,4$ and the sinks $1,3$, and in the second one $0,2\zp/3,4\zp/3$
and $\zp/3,\zp,5\zp/3$ respectively.

When $\dim B\zmai 2$ we have created new singularities called artificial. Now
instead of that one will increase the order of sinks (otherwise the non-singular
set has too many components). Let $\zt\colon B\zfl \mathbb R$ be  a bounded
function which is positive outside sinks and has order two at $1$ and
$\zp/3$, order four at $3$ and $\zp$ and, finally, order six at
$5\zp/3$. Set $Z=\zt Y$, which is a complete vector field.

This time the $Z$-outsets $\{R_i \}_{i\zpe I}$, where $I=\{0,2,4\}$ if
$B=\mathbb R$ or $I=\{0,2\zp/3,4\zp/3\}$ if $B=S^1$, equal those
of $Y$, and any of them is an invariant of $Z$ because the $\zw$-limits
of its trajectories have different orders or are empty; even more,
every side of the outset has to be preserved.

On $M=B\zpor\mathbb T^n$ with coordinates $(x,\zh)$ or $(\za,\zh)$ set
$X'=\zt(Y+T)$, where $T$ is a dense affine vector field and $\zt, Y, T$
are regarded on $M=B\zpor\mathbb T^n$ in the natural way.
Now $(X')^{-1}(0)=(B-\zung_{i\zpe I}R_i )\zpor\mathbb T^n$. Therefore
$(X')^{-1}(0)$ is the union of two ($B=\mathbb R$) or three ($B=S^1$)
fibres; moreover points of different fibres have different orders.

Let $F\colon M\zfl M$ be an automorphism of $X'$; reasoning as in the case
$\dim B\zmai 2$ shows that $F=\zl_i \zci \zF_{t_i}$ on
$R_i \zpor\mathbb T^n$ for each $i\zpe I$, where $\zF_t$ is the flow of
$X'$. Thus the induced automorphism $f\colon B\zfl B$ of $Z$ equals
$\zf_{t_i}$ on $R_i$, where $\zf_t$ is the flow of $Z$. Now Lemma
\ref{lemG} implies that $t_i =t_j$ if $R_i$ and $R_j$ are contiguous. In short
$f=\zf_t$ for some $t\zpe\mathbb R$. The remainder of the proof is
similar to that of $\dim B\zmai 2$.

Finally notice that $\zr X'$ is a describing vector field too if
$\zr\colon M\zfl\mathbb R$ is $\mathbb T^n$-invariant, positive and
bounded; therefore the result can be extended to the case of an effective ç
action following the lines in Section \ref{secC}.

\section{Examples}\label{secE}

One starts this section by giving two examples of effective toric actions. The
first one is a general construction on the Lie groups. In the second example
a describing vector field is constructed for the usual action of $\mathbb T^3$
on $S^3$.

On the other hand two examples more show that the main theorem fails for
general compact Lie groups. More exactly for effective actions of
$SO(3)$ (Example \ref{ejeB}) and for effective actions of a non-connected
compact group, of dimension two, with abelian Lie algebra (Example \ref{ejeC}).

\begin{example}\label{ejeAcero}
{\rm Let $G$ be a connected Lie group, with center $ZG$, and
two (non necessary maximal) tori $H,\widetilde H\leq G$. Then there exists a
$(H\zpor\widetilde H)$-action on $G$ given by
$(h,\tilde h )\zpu g= hg{\tilde h}^{-1}$, whose kernel $K$ equals
$\{(h,h)\zbv h\zpe H\zin\widetilde H\zin ZG\}$. Thus an effective action
 of the torus $(H\zpor\widetilde H)/K$ on $G$ is induced.

Now suppose that $G$ is compact with rank $r$ and $ZG$ finite,
that is the center of the Lie algebra of $G$ is zero. Let $H$ be a
maximal torus of $G$; set  $\widetilde H=H$. Then one obtains
an effective action of  $\mathbb T^{2r}\zeq (H\zpor H)/K$ on $G$.
 Moreover the isotropy group of any
$g\zpe G$ has two or more elements if and only if
$(gHg^{-1})\zin H$ is not included in $ZG$; by Proposition \ref{proAA}
 this happens for almost no $g\zpe G$.}
\end{example}

\begin{example}\label{ejeA}
{\rm On $S^5 =\{y\zpe \mathbb R^6 \zbv y_{1}^2 +\dots+y_{6}^2 =1\}$
consider the action of $\mathbb T^3$ given by the fundamental vector fields
$U_j =-y_{2j}\zpar/\zpar y_{2j-1}+y_{2j-1}\zpar/\zpar y_{2j}$,
$j=1,2,3$. In order to construct a describing vector field for this action we follow along the lines of
Sections \ref{secB} and \ref{secC} up to some minor changes.
First observe that the singular set for this action is
$S=\{y\zpe S^5 \zbv (y_{1}^2 +y_{2}^2)(y_{3}^2 +y_{4}^2)
(y_{5}^2 +y_{6}^2)=0\}$, so the action of $\mathbb T^3$ on $S^5 -S$
is free.

Let $\zp\colon S^5 \zfl\mathbb R^2$ be the map given by
$\zp(y)=(y_{1}^2 +y_{2}^2 ,y_{3}^2 +y_{4}^2 )$, and $B$ be the interior of the
triangle of vertices $(0,0),(1,0),(0,1)$. Then
$\zp\colon S^5 -S\zfl\mathbb B$ is the $\mathbb T^3$-principal bundle associated
to the action of $\mathbb T^3$. A connection $\mathcal C$ for this principal bundle
is defined by $Ker(\za_1 \zex\za_2 \zex\za_3 )$ where each
$\za_j =(y_{2j-1}^2 +y_{2j}^2 )^{-1}(-y_{2j}dy_{2j-1}+y_{2j-1}dy_{2j})$,
which is flat since $d\za_1 =d\za_2 =d\za_3 =0$.

The vector fields
$$V_r =(y_{5}^2 +y_{6}^2 )\left(  y_{2r-1}{\frac {\zpar} {\zpar y_{2r-1}}}
+y_{2r}{\frac {\zpar} {\zpar y_{2r}}}\right)
-(y_{2r-1}^2 +y_{2r}^2 )\left(  y_{5}{\frac {\zpar} {\zpar y_{5}}}
+y_{6}{\frac {\zpar} {\zpar y_{6}}}\right)\, ,$$
$r=1,2$, are tangent to $\mathcal C$ and project in
$$2(1-x_{1}-x_{2}){\frac {\zpar} {\zpar x_{r}}}\, ,$$
$r=1,2$, on $B$ endowed with coordinates $x=(x_1 ,x_2)$.

Set
$$Y=2(1-x_{1}-x_{2})x_1 x_2 \left[ \left( x_1 -{\frac {1}  {4}}\right)
{\frac {\zpar} {\zpar x_{1}}}+ \left( x_2 -{\frac {1}  {4}}\right)
{\frac {\zpar} {\zpar x_{2}}}\right]\, ,$$
whose lifted vector field through $\mathcal C$ is

$$Y'= (y_{1}^2 +y_{2}^2)(y_{3}^2 +y_{4}^2)
\left(( y_{1}^2 +y_{2}^2 -1/4)V_1 +
( y_{3}^2 +y_{4}^2 -1/4)V_2\right).$$

Note that $Y'$ extends in a natural way to $S^5$. Moreover
$Y'$ vanishes on $S$.

Clearly $Y$ on $B$ and $Y'$ on $S^5$ and $S^5 -S$ are complete.
Besides $(1/4,1/4)$ is the only source of $Y$.

Let $\zt\colon\mathbb R^2 \zfl\mathbb R$ be the function defined by
$$\zt(x)=\zr(x)\left((x_1 -1/8)^{2}+(x_2 -1/8)^{2}\right)
\left((x_1 -1/8)^{2}+(x_2 -1/4)^{2}\right)^2
\left((x_1 -1/4)^{2}+(x_2 -1/8)^{2}\right)^3$$
where $\zr(x)=x_{1}^{10}x_{2}^{10}(1-x_{1}-x_{2})^{10}$, whose
zeros on in $B$ are $(1/8,1/8)$ with order two,  $(1/8,1/4)$ with order four
and  $(1/4,1/8)$ with order six.

Now $X'=(\zt\zci\zp)(Y'+U_1 +e U_2 +e^2 U_3)$, defined on $S^5$, is
a describing vector field for the action of $\mathbb T^3$.

Indeed, the only difference with respect to the construction of Sections
\ref{secB} and \ref{secC} is that every point of $S$ is a singularity of
$X'$ with order $\zmai 10$ instead of infinity, but it is not important because
the order of the remainder singularities of $X'$ is always $\zmei 6$.}
\end{example}

\begin{example}\label{ejeB}
{\rm Let $H$ be a closed subgroup of a connected Lie group $G$ and
$G/H$ be the (quotient) homogeneous space associated to the equivalence
relation $g_1 \mathcal R g_2$ if and only of $g_2 =g_1 h$ for some
$h\zpe H$. As it is well known $G$ acts on $G/H$ by setting
$g\zpu\overline{ g'}=\overline{gg'}$.

Now assume $H$ discrete; then the canonical projection
$\zp\colon G\zfl G/H$ is a covering. Moreover a vector field $V$ on $G/H$
commutes with the action of $G$, that is to say with every fundamental
vector field, if and only if its lifted vector field $V'$, on $G$, is at the same
time left $G$-invariant and right $H$-invariant.

If one suppose that $V'$ is left $G$-invariant this property is equivalent
to say that $V'(e)$, where $e$ is the identity of $G$, is invariant by
(the adjoint action of) $H$. Therefore if no element of $T_e G-\{0\}$
is invariant by $H$, then any vector field on $G/H$ which commutes with
the action of $G$ identically vanishes.

Set $G=SO(3)$ and let $H$ be the subgroup of order four consisting of the
identity plus the three rotations of angle $\zp$ around any of the coordinates
axes (i.e.\ $H$ is the Klein four-group). Then no element of  $T_e SO(3)-\{0\}$ is invariant by $H$.

Consider on $M=\mathbb R^k \zpor(SO(3)/H)$, $k\zmai 1$, the action of
$SO(3)$ given by $g\zpu(x,\overline{g'})=(x,\overline{gg'})$. {\it This action is
effective but it does not have any describing vector field.} Indeed, take a vector
field $U$ on $M$ which commutes with the action of $SO(3)$. Then $U$ has
to respects the foliation defined by the orbits of $SO(3)$, so
$U=\zsu_{j=1}^k f_j (x)\zpar/\zpar x_j +V$ where $V$ is a vector field
tangent to the second factor and $x=(x_1 ,\dots,x_k )$ the canonical
coordinates in $\mathbb R^k$.

Since $U$ and $\zsu_{j=1}^k f_j (x)\zpar/\zpar x_j$ commute with the
action, $V$ has to commute with the fundamental vector
fields on each orbit, hence $V=0$. In other words $U=\zsu_{j=1}^k f_j (x)\zpar/\zpar x_j$.

On the other hand if $\zf\colon SO(3)/H\zfl SO(3)/H$ is a diffeomorphism, then
$$F\colon\mathbb R^{k}\zpor(SO(3)/H)\zfl\mathbb R^k \zpor(SO(3)/H)$$ given
by $F(x,\overline{g})=(x,\zf(\overline{g}))$ is an automorphism of $U$, so
$\Aut(U)$ is strictly greater than $SO(3)\zpor\mathbb R$.

Another possibility is to consider the action of $SO(3)$ on the sphere $S^2$.
Then for each $p\zpe S^2$ there exists a fundamental vector field $X$ such
 that $p$ is an isolated singularity of $X$. Therefore if $V$ commutes with $X$
then $V(p)=0$; consequently if $V$ commutes with the action of $SO(3)$ on
$S^2$ necessarily $V=0$.

By the same reason as before, the action of $SO(3)$ on
$\mathbb R^k \zpor S^2$, $k\zmai 1$, defined by $g\zpu(x,p)=(x,g\zpu p)$
has no  describing vector field.}
\end{example}

\begin{example}\label{ejeC}
{\rm Let $G$ be the group of affine transformations $\zf\colon\mathbb T^2
\zfl\mathbb T^2$ defined by $\zf(\zh)=a\zh+\zl$ where
$a=\zmm 1$ and $\zl\zpe\mathbb T^2$.
The fundamental vector fields of the natural action
of $G$ on $\mathbb T^2$ are $b_1 \zpar/\zpar\zh_1
+b_2 \zpar/\zpar\zh_2$, $b_1 ,b_2 \zpe\mathbb R$.
If $V$ is a vector field on $\mathbb T^2$ which commutes with the action of $G$
it has to commute with the fundamental vector fields, therefore
$V=c_1 \zpar/\zpar\zh_1  +c_2 \zpar/\zpar\zh_2$, $c_1 ,c_2 \zpe\mathbb R$.

But, at the same time, $V$ commutes with $\tilde\zf(\zh)=-\zh$
since $\tilde\zf\zpe G$, which implies $V=0$.
Now reasoning as in Example \ref{ejeB} shows that the effective action of
$G$ on $\mathbb R^k \zpor\mathbb T^2$, $k\zmai 1$, defined by
$\zf\zpu(x,\zh)=(x,\zf(\zh))$ has no describing vector field.

Observe that $G$ is a compact Lie group with abelian Lie algebra, but
it is not connected.}
\end{example}

\section{Two auxiliary results}\label{secAA}

Here we include two complementary results which were
needed before. The first one is a straightforward consequence of the Principal Orbit Theorem
on the structure of the orbits of the action of a compact Lie group (see
\cite[Theorem IV.3.1]{bredon}).

\begin{proposition}\label{proAA}
Consider an effective action of $\mathbb T^n$ on a connected
$m$-manifold $M$. Let $S$ be the set of those points of $M$ whose
isotropy group has two or more elements, i.e.\ the isotropy group is non trivial. Then the set $M-S$ is connected,
dense, open and $\mathbb T^n$-invariant. Moreover $n\zmei m$.
\end{proposition}

The second result shows how, for connected compact Lie group actions,
locally defined invariant vector fields give rise to invariant
vector fields defined on the whole manifold.

\begin{proposition}\label{proAB}
Consider an action of a connected compact Lie group $G$ on a
$m$-manifold $M$. Given a vector field $X$ on an open set $A$ of $M$,
both of them $G$-invariant, then there exists a $G$-invariant
bounded function $\zf:M\zfl{\mathbb R}$,  which is positive on $A$ and
vanishes on $M-A$, such that the vector field ${\hat X}$ on $M$ defined
 by ${\hat X}=\zf X$ on $A$ and ${\hat X}=0$ on $M-A$ is differentiable
and $G$-invariant.
\end{proposition}

First let us recall some elementary facts on actions. Let $Z$ be a vector
field on $M$ and $\widetilde Z$ the vector field depending on a parameter
$g\zpe G$ given by $\widetilde Z(g,p)=(g_*)^{-1}(Z(g\zpu p))$. On $G$
consider a bi-invariant volume form whose integral equals $1$ and
 the measure associated to it, that is the Haar measure. Then since
$\widetilde Z(\{p\}\zpor G)\zco T_p M$ the formula
$$Z'(p)=\zil_G \widetilde Z(g,p)$$
defines a $G$-invariant vector field $Z'$ on $M$. Moreover if $Z=\zr U$
where $U$ is a $G$-invariant vector field, then $Z'=\zq_\zr U$ where
$\zq_\zr$ is the $G$-invariant function constructed from $\zr$ in the usual way,
that is $$\zq_\zr(p)=\zil_G \zr(g\zpu p)\, .$$

In order to prove Proposition \ref{proAB}, we start considering, for $X$ and $A$,
a function $\zf\colon M\zfl\mathbb R$ like in
\cite[Proposition 5.5, pag.\ 329]{TV}, and the vector field $\hat X$ defined by $\hat X=\zf X$ on $A$
and $\hat X=0$ on M-$A$.

Now observe that $\hat X'=\zq_\zf X$ on $A$ because on this open set
$X$ is $G$-invariant and $\hat X=\zf X$. Thus $\zq_\zf$ is the required
function  (up to the name) since $\zq_\zf $ and $\hat X'$ vanish on
$M-A$; so {\it Proposition \ref{proAB} is proved}.


\end{document}